\documentclass[12pt, reqno]{amsart}
\usepackage [latin1]{inputenc}
\usepackage{amscd}
\usepackage{epsfig}
\usepackage{amssymb}
\usepackage{amsmath}
\usepackage{amsthm}
\usepackage[T1]{fontenc}
\usepackage{ae,aecompl}

\DeclareMathOperator{\dom}{dom}

\usepackage {amsfonts} 
\usepackage {latexsym} 
\usepackage {exscale} 
\usepackage{amsmath}

\newcommand{\R} {\ensuremath{\mathbb{R}}}

\newcommand{\C} {\ensuremath{\mathbb{C}}}

\renewcommand{\o}[1]{\overline{#1}}

\newcommand{\dq}{\overline{\partial}}

\newtheorem {satz} {Satz} [section]
\newtheorem {lem} [satz] {Lemma}

\newtheorem {defn} [satz] {Definition}

\newtheorem {thm} [satz] {Theorem}

\DeclareMathOperator{\dist}{dist}
\DeclareMathOperator{\supp}{supp}

\renewcommand{\theta}{\vartheta}

\title[Friedrichs' extension lemma with boundary values] 
{Friedrichs' extension lemma with boundary values and applications in complex analysis }

\author{J. Ruppenthal}

\address{Department of Mathematics, University of Wuppertal, Gau{\ss}str. 20, 42119 Wuppertal, Germany.}
\email{ruppenthal@uni-wuppertal.de}

\date{October 14, 2009}

\subjclass[2000]{32W05, 35F99, 47F05}
\keywords{Weak boundary values, generalized Stokes' formulas, Cauchy-Riemann operator}

\begin{document}

\begin{abstract} 
Let $Q$ be a first-order differential operator on a compact, smooth oriented Riemannian manifold
with smooth
boundary. Then, Friedrichs' extension lemma states that the minimal closed
extension $Q_{min}$ (the closure of the graph) and the maximal closed extension $Q_{max}$ 
(in the sense of distributions) of $Q$ in $L^p$-spaces ($1\leq p<\infty$) coincide. 
In the present paper, we show that the same is true for boundary values with respect to $Q_{min}$ and $Q_{max}$.
This gives a useful characterization of weak boundary values, particularly for $Q=\dq$
the Cauchy-Riemannn operator.
As an application, we derive the Bochner-Martinelli-Koppelman formula
for $L^p$-forms with weak $\dq$-boundary values.
\end{abstract}

\maketitle

\section{Introduction}

Let $D$ be a relatively compact domain in a Hermitian complex manifold
and $\dq: C^\infty_*(\o{D})\rightarrow C^\infty_*(\o{D})$ the Cauchy-Riemann operator
on smooth forms. For $1\leq p < \infty$, this operator can be considered
as a densely defined graph-closable operator on $L^p$-forms:
$$\dq: \dom(\dq)=C^\infty_*(\o{D}) \subset L^p_*(D)\rightarrow L^p_*(D)$$
Now then, the $\dq$-operator has various closed extensions. The two most important
are the minimal closed extension $\dq_{min}$ given by the closure of the graph
and the maximal closed extension $\dq_{max}$, i.e. the $\dq$-operator in the sense
of distributions. 
Whereas the two extensions coincide on smoothly bounded domains by Friedrichs'
extension lemma (see \cite{F}, \cite{Hoe}),
one has to be very careful when considering non-smooth domains.
Especially on regular sets in singular complex spaces, it is crucial to distinguish
the different closed extensions of the $\dq$-operator for they lead to
different Dolbeault cohomology groups (see e.g. \cite{BS}, \cite{P}, \cite{PS1} or \cite{PS2}).
It was realized that investigating the relation between the various extensions
is an essential and very fruitful (maybe even indispensable) step in understanding the $\dq$-equation on singular
complex spaces which has to be pursued (see also \cite{R2}).
Clearly, the difference between the closed extensions occurs at the boundary of the domain.
So, a first step is to study the boundary behavior 
of $\dq_{min}$ and $\dq_{max}$ on domains with smooth boundary
which we do in the present paper by deriving a Friedrichs' extension lemma with boundary values.

Let $D\subset\subset \C^n$ be a bounded domain with smooth boundary $bD$, and let
$f\in L_{0,q}^p(D)$ with $\dq f\in L^p_{0,q+1}(D)$ in the sense of distributions for $1\leq p<\infty$.
Then, we say that $f$ has weak $\dq$-boundary values $f_b\in L^p_q(bD)$ in the sense of distributions if
\begin{eqnarray}\label{eq:intro1}
\int_D \dq f\wedge \phi + (-1)^q \int f\wedge \dq \phi = \int_{bD} f_b\wedge \iota^*(\phi)
\end{eqnarray}
for all $\phi\in C^\infty_{n,n-q-1}(\o{D})$, where $\iota: bD\hookrightarrow \C^n$
is the embedding of the boundary.
Weak $\dq$-boundary values in the sense of distributions
are a classical subject of complex analysis (see Theorem \ref{thm:hapo}, for example)
and closely related to the investigation
of the so-called Hardy spaces (cf. \cite{Sk}).
Starting from results of Skoda \cite{Sk}, Harvey and Polking \cite{HaPo},
Schuldenzucker \cite{Sch} and Hefer \cite{He2},
there has been a considerable progress in the understanding of weak $\dq$-boundary values
by Hefer in \cite{He3},
where boundary values in the sense of distributions are compared to
boundary values which arise naturally in the application of integral operators.
This is interesting because boundary values
defined by restricting the kernel of an integral operator can often be estimated by direct
methods, whereas the abstractly given distributional boundary values are less tractable but
analytically interesting objects linked to the form on the interior of a domain.

However, in applications the definition of weak $\dq$-boundary values by means of the Stokes'
formula \eqref{eq:intro1} turns out to be a bit unhandy and it is more convenient to
have boundary values in the sense of approximation by smooth forms.
In fact, let $f\in \dom(\dq_{max})\subset L^p_{0,q}(D)$ with weak boundary values $f_b\in L^p_q(bD)$
according to definition \eqref{eq:intro1}, and let $r\in C^\infty(\C^n)$ be a smooth defining function for $D$.
Then we will show that there exists a sequence $f_j\in C^\infty_{0,q}(\o{D})$
such that
\begin{eqnarray}\label{eq:intro2}
f_j\rightarrow f \mbox{ in } L^p_{0,q}(D)\ ,\ \ \ \dq f_j \rightarrow \dq f \mbox{ in } L^p_{0,q+1}(D)
\end{eqnarray}
(the classical Friedrichs' extension lemma) and moreover
\begin{eqnarray}\label{eq:intro3}
f_j\wedge \partial r \rightarrow f_b \wedge \partial r \mbox{ on } bD \mbox{ in } L^p_q(bD),
\end{eqnarray}
i.e. $f$ has $\dq$-boundary values in the sense of approximation (Theorem \ref{thm:main2}).

This phenomenon is not restricted to the Cauchy-Riemann operator, but holds
for arbitrary differential operators of first order with smooth coefficients. So, it is more convenient
to adopt a more general point of view.
Let $M$ be a smooth, compact Riemannian manifold with smooth boundary, 
$E$ and $F$ Hermitian vector bundles over $M$, and
$Q: C^\infty(M,E) \rightarrow C^\infty(M,F)$
a differential operator of first order. Let $1\leq p<\infty$ and $f\in L^p(M,E)$.
We say that 
$f\in \dom(Q_{min}^p)$
if there exists a sequence $\{f_j\}\subset C^\infty(M,E)$ and a section $g\in L^p(M,F)$
such that
$$f_j\rightarrow f \mbox{ in } L^p(M,E)\ ,\ \ \ Qf_j\rightarrow g \mbox{ in } L^p(M,F),$$
and define $Q_{min}^p f:=g$.
The well-defined operator $Q_{min}^p$ is called the minimal extension of $Q$
because it is the closed extension of $Q$ to an operator \linebreak $L^p(M,E)\rightarrow L^p(M,F)$ 
with minimal domain of definition.
Its graph is simply the closure of the graph of 
$Q: C^\infty(M,E) \rightarrow C^\infty(M,F)$ in $L^p(M,E)\times L^p(M,F)$.
Let $\sigma_Q$ be the principal symbol of $Q$,
$\nu$ the outward pointing unit normal to $bM$, and $\nu^\flat$
the dual cotangent vector.
Then, we say that $f$ has boundary values with respect to $Q_{min}^p$
if there exists a sequence $\{f_j\}$ in $C^\infty(M,E)$ such that 
$\lim_{j \rightarrow \infty} f_j=f$ in $L^p(M,E)$,
$\lim_{j \rightarrow \infty} Qf_j = Q^p_{min} f$ in $L^p(M,F)$, and
a section $f_b\in L^p(bM,E|_{bM})$ such that
$$\lim_{j \rightarrow \infty} \sigma_Q(\cdot,\nu^\flat(\cdot))f_j|_{bM} 
= \sigma_Q(\cdot,\nu^\flat(\cdot))f_b\ \ \mbox{ in }\ L^p(bM,F|_{bM}).$$
In this case, we call $f_b$ weak $Q$-boundary values of $f$
with respect to $Q^p_{min}$ (i.e. in the sense of approximation).

Now, we draw our attention to the maximal closed extension of $Q$,
that is the extension of $Q$ in the sense of distributions.
We say that $f\in \dom(Q^p_{max})$ if $Qf=u\in L^p(M,F)$ in the sense of distributions,
and set $Q^p_{max} f:= u$ in that case. Here again, we can define weak $Q$-boundary values
with respect to $Q^p_{max}$.
We say that $f$ has weak $Q$-boundary values 
$f_b\in L^p(bM, E|_{bM})$ with respect to $Q^p_{max}$
(in the sense of distributions), if $f_b$ satisfies
the generalized Green-Stokes formula (cf. Theorem \ref{thm:stokes})
\begin{eqnarray*}
(Qu,\phi)_M - (u,Q^*\phi)_M = \frac{1}{i}\int_{bM} \langle \sigma_Q(x,\nu^\flat) u_b, \phi\rangle_{F_x}\ dS(x)
\end{eqnarray*}
for all $\phi\in C^\infty(M,F)$.

The main objective of the present paper is to compare both notions
of $Q$-boundary values. It is easy to see that $\dom(Q^p_{min}) \subset \dom(Q^p_{max})\subset L^p(M,E)$,
and that $Q^p_{min}$ is the restriction of $Q^p_{max}$ to $\dom(Q^p_{min})$.
Moreover, it is also clear that weak $Q$-boundary values in the sense of approximation
are weak $Q$-boundary values in the sense of distributions as well.
It is well-known that in fact $Q^p_{min}=Q^p_{max}$ on smooth, compact manifolds
with smooth boundary.
This result, due to Friedrichs (see \cite{F}, \cite{Hoe}), is usually called Friedrichs' extension lemma (Theorem \ref{thm:friedrichs}).
In this paper, we observe that the two notions of boundary values coincide as well (Theorem \ref{thm:main}).
One might call this Friedrichs' extension lemma with boundary values.
In the particular case of the Cauchy-Riemann operator $Q=\dq$, we obtain \eqref{eq:intro2}, \eqref{eq:intro3}.

The present paper is organized as follows:
in the next section, we recall the notion of weak $Q$-boundary values in the sense of distributions (Definition \ref{defn:bvalues})
which makes sense in view of the generalized Green-Stokes formula Theorem \ref{thm:stokes}.
In section \ref{sec:fel}, we recall the proof of the classical Friedrichs' extension lemma
as it is presented in \cite{LiMi} (relying on \cite{Hoe} which in turn cites \cite{F} and \cite{LP})
and prove Friedrichs' extension lemma with boundary values by a sophisticated choice of an approximating identitiy.
In section \ref{sec:dq}, we return to the Cauchy-Riemann operator by translating 
the results into the language of complex analysis in the particular case of the differential operator $Q=\dq$.
In the last section, we show how boundary values in the sense of approximation can be used in applications
by deriving the Bochner-Martinelli-Koppelman formula for forms with weak $\dq$-boundary values.

\section{Weak Boundary Values}

Let $M$ be a smooth, compact Riemannian manifold with smooth boundary, 
$E$ and $F$ Hermitian vector bundles over $M$, and
$$Q: C^\infty(M,E) \rightarrow C^\infty(M,F)$$
a differential operator of first order.
Let $\sigma_Q$ be the principal symbol of $Q$, and
$$Q^*: C^\infty(M,F)\rightarrow C^\infty(M,E)$$
its formal adjoint operator given by
$$(Qu,v)_M = \int_M \langle Qu,v\rangle_F\ dV_M = \int_M \langle u,Q^* v\rangle_E\ dV_M=(u,Q^*v)_M,$$
where one of the two sections $u\in C^\infty(M,E)$, $v\in C^\infty(M,F)$ has compact support in 
the interior of $M$. Let $dS$ be the induced volume element on the boundary
$b M$, $\nu$ the outward pointing unit normal to $bM$, and $\nu^\flat$
the dual cotangent vector.
Then, the generalized Green-Stokes formula reads as (see \cite{Ta}, Prop. 9.1):

\begin{thm}\label{thm:stokes}
Let $M$ be a smooth, compact Riemannian manifold with smooth
boundary, and $Q$ a first-order differential operator (acting on sections of Hermitian vector bundles). Then
$$(Qu,v)_M-(u,Q^*v)_M = \frac{1}{i}\int_{b M} \langle\sigma_Q(x,\nu^\flat)u,v\rangle_{F_x} \ dS(x)$$
for all sections $u\in C^\infty(M,E)$, $v\in C^\infty(M,F).$
\end{thm}

Now, let $u\in L^1(M,E)$ and $f\in L^1(M,F)$. Then we say that $Qu=f$ in the sense
of distributions if
$$(u,Q^*\phi)_M = (f, \phi)_M$$
for all $\phi\in C^\infty(M,F)$ with compact support in the interior of $M$.
We can now give the definition of weak boundary values with respect to the
first-order differential operator $Q$:

\begin{defn}\label{defn:bvalues}
In the situation of Theorem \ref{thm:stokes}, let $u\in L^1(M,E)$
with $Qu\in L^1(M,F)$. Then $u$ has weak $Q$-boundary values 
$u_b\in L^p(bM, E|_{bM})$ if
\begin{eqnarray}\label{eq:bvalues}
(Qu,\phi)_M - (u,Q^*\phi)_M = \frac{1}{i}\int_{bM} \langle \sigma_Q(x,\nu^\flat) u_b, v\rangle_{F_x}\ dS(x)
\end{eqnarray}
for all $\phi\in C^\infty(M,F)$.
\end{defn}

This generalizes the notion of weak boundary values of functions
in the Sobolev space $H^{1,p}(M)$: Let $Q=d:\ C^\infty(M,\C) \rightarrow C^\infty(M,\C\otimes T^* M)$
be the exterior derivative.
Then, for all $1\leq p\leq \infty$, there is a unique continuous trace operator
$$T: H^{1,p}(M)=\{u\in L^p(M,\C): du\in L^p(M,\C\otimes T^* M)\}\rightarrow L^p(bM,\C)$$
such that $Tu$ satisfies \eqref{eq:bvalues} (cf. \cite{Alt}, A 6.6).
In general, weak $Q$-boundary values do not necessarily exist.

\section{Friedrichs' Extension Lemma}\label{sec:fel}

Again,
let $M$ be a smooth, compact Riemannian manifold with smooth boundary, 
$E$ and $F$ Hermitian vector bundles over $M$, and
$Q: C^\infty(M,E) \rightarrow C^\infty(M,F)$
a differential operator of first order (with $C^1$ coefficients). Let $1\leq p<\infty$.
Then, for $f\in L^p(M,E)$, we say that
$f\in \dom(Q_{min}^p)$
if there exists a sequence $\{f_j\}\subset C^\infty(M,E)$ and $g\in L^p(M,F)$
such that
$$f_j\rightarrow f \mbox{ in } L^p(M,E)\ ,\ \ \ Qf_j\rightarrow g \mbox{ in } L^p(M,F),$$
and define $Q_{min}^p f:=g$.
The operator $Q_{min}^p$ is uniquely defined, because
$$(g,h)_M = \lim_{j\rightarrow \infty} (Qf_j,h)_M = \lim_{j\rightarrow\infty} (f_j,Q^* h)_M= (f,Q^* h)_M$$
for all $h\in C^\infty(M,F)$ with compact support in the interior of $M$.
Moreover, we say that $f\in \dom(Q^p_{max})$, if $Qf=u\in L^p(M,F)$ in the sense of distributions,
and set $Q^p_{max} f:= u$ in that case. It is easy to see that
$$\dom(Q^p_{min}) \subset \dom(Q^p_{max})\subset L^p(M,E),$$ 
and $Q^p_{min}$ is the restriction of $Q^p_{max}$ to $\dom(Q^p_{min})$.
But, in our situation, also the converse is true (cf. \cite{LiMi}, Theorem V.2.6):

\begin{thm}{\bf (Friedrichs' Extension Lemma)}\label{thm:friedrichs}
Let $M$ be a smooth, compact Riemannian manifold with smooth
boundary, and $Q: C^\infty(M,E)\rightarrow C^\infty(M,F)$ 
a first-order differential operator (acting on sections of Hermitian vector bundles),
and $1\leq p<\infty$.
Then for any $f\in \dom(Q^p_{max})$ there exists a sequence $\{f_\epsilon\}$ in $C^\infty(M,E)$
such that $\lim_{\epsilon\rightarrow 0} f_\epsilon=f$ and 
$\lim_{\epsilon\rightarrow 0} Qf_\epsilon=Q^p_{max} f$
with respect to $L^p$-norms. Shortly this means that
$$Q^p_{min}=Q^p_{max}.$$
\end{thm}

Let us recall the principles of the proof. Using a partition of unity,
it is enough to consider $U\subset\subset \R^n$ open with smooth boundary
and $Q: C^\infty(U)\rightarrow C^\infty(U)$.
So, let $f\in L^p(U)$ and $Qf=Q^p_{max} f\in L^p(U)$. Again, by the partition of unity argument,
one has to consider the following two cases:\hfill\vspace{2mm}\linebreak
{\bf 1.} $\supp(f) \subset\subset U$, or\hfill\vspace{1mm}\linebreak
{\bf 2.} $U=\{x\in\R^n: x_1<0\}$ and $\supp(f)\subset\subset \o{U}$.\hfill\vspace{2mm}\linebreak
For the first case, let $\phi\in C^\infty_{cpt}(B_1(0))$ with $\phi\geq 0$ and $\int\phi dx=1$,
where $dx$ is the Euclidean volume element. We call
$\phi_\epsilon(x):=\epsilon^{-n}\phi(x/\epsilon)$ a Dirac sequence, and
$$f_\epsilon := f * \phi_\epsilon$$
the convolution of $f$ with a Dirac sequence. It is well known that $f_\epsilon \rightarrow f$
in $L^p(U)$ for $\epsilon\rightarrow 0^+$. But the crucial observation is
\begin{lem}\label{lem:lp}
$$\|Q f_\epsilon - (Q f)* \phi_\epsilon\|_{L^p(U)} \lesssim \|f\|_{L^p(U)}.$$
\end{lem}
It is now easy to complete the first case: Let $\delta>0$ and $\psi\in C^\infty_{cpt}(U)$
such that
$$\|f-\psi\|_{L^p(U)}<\delta.$$
Applying Lemma \ref{lem:lp} to $f-\psi$ yields:
$$\|Qf_\epsilon - (Qf)*\phi_\epsilon\|_{L^p(U)} 
\lesssim \delta + \|Q\psi_\epsilon - (Q\psi)*\phi_\epsilon\|_{L^p(U)}.$$
Choosing $\delta$ and $\epsilon$ arbitrarily small finishes this part of the proof.
The second case is treated by exactly the same procedure.
One only has to be a little careful when choosing the Dirac sequence $\phi_\epsilon$.
Here, let $\phi\in C^\infty_{cpt}(B_1(0))$ such that
$$\supp(\phi)\subset\subset \{x\in B_1(0): x_1>0\}.$$
Then $f_\epsilon$ is well defined on $\o{U}$, 
Lemma \ref{lem:lp} is still true and everything goes through as before.
That completes the proof of Theorem \ref{thm:friedrichs} as it is given in \cite{LiMi}.

\vspace{2mm}
We are now interested in the behavior of the sequence $\{f\epsilon\}$ on the boundary $bM$.
It is possible to extend Theorem \ref{thm:friedrichs} to Friedrichs' extension lemma
with boundary values:

\begin{thm}\label{thm:main}
In the situation of Theorem \ref{thm:friedrichs}, assume that $f\in \dom(Q^p_{max})$
has weak $Q$-boundary values $f_b\in L^p(bM, E|_{bM})$ in the sense of
Definition \ref{defn:bvalues}.
Then there exists a sequence $\{f_\epsilon\}$ in $C^\infty(M,E)$ such that 
$\lim_{\epsilon\rightarrow 0} f_\epsilon=f$ in $L^p(M,E)$,
$\lim_{\epsilon\rightarrow 0} Qf_\epsilon = Q^p_{max} f$ in $L^p(M,F)$ and
$$\lim_{\epsilon\rightarrow 0} \sigma_Q(\cdot,\nu^\flat(\cdot))f_\epsilon|_{bM} 
= \sigma_Q(\cdot,\nu^\flat(\cdot))f_b\ \ \mbox{ in }\ L^p(bM,F|_{bM}).$$
\end{thm}

\begin{proof}
We copy the proof of Theorem \ref{thm:friedrichs}. One has to be
even more careful when choosing the Dirac sequence.
We only have to take a closer look at the second case. So, let
$U=\{x\in \R^n: x_1<0\}$, $\supp(f)\subset\subset \o{U}$
and 
$$\supp(f_b)\subset\subset bU=\{x\in\R^n: x_1=0\}.$$
Then
$$Q=\sum_{j=1}^n a_j(x)\frac{\partial}{\partial x_j} + b(x)\ ,\ \ 
Q^*=-\sum_{j=1}^n \left(a_j(x) \frac{\partial}{\partial x_j} + \frac{\partial a_j}{\partial x_j}(x)\right) + b(x),$$
and
\begin{eqnarray}\label{eq:bvalues1}
\int_U (Qf) \Phi dx - \int_U f (Q^* \Phi) dx = \int_{bU} a_1(0,x') f_b(x') \Phi(0,x') dx'
\end{eqnarray}
for all $\Phi\in C^\infty_{cpt}(\o{U})$ according to Definition \ref{defn:bvalues} with $x'=(x_2, ..., x_n)$.
We will use the decomposition
\begin{eqnarray}\label{eq:Q}
Q^* = -a_1(x)\frac{\partial}{\partial x_1} + Q'.
\end{eqnarray}
Now, let us choose the right Dirac sequence for our purposes.
Let $B'_1(0)$ be the unit ball in $\R^{n-1}$ and $\psi\in C^\infty_{cpt}(B'_1(0))$
with $\psi\geq 0$ and $\int \psi dx'=1$, where $dx'$ is the Euclidean volume element in $\R^{n-1}$,
$x'=(x_2, ..., x_n)$. For $\epsilon>0$ set $$\psi_\epsilon:=\epsilon^{-(n-1)}\psi(x'/\epsilon).$$
Therefore, it follows that 
\begin{eqnarray}\label{eq:bvalues2}
\lim_{\epsilon\rightarrow 0}\ \big(a_1(0,\cdot) f_b\big) * \psi_\epsilon 
=a_1(0,\cdot) f_b\ \ \mbox{ in }\ L^p(bU).
\end{eqnarray}
Moreover, let $h: \R \rightarrow [0,1]$
be a smooth function such that 
$$h(x_1) = \left\{\begin{array}{ll}
0,& \mbox{ for } x_1\leq 1,\\
1,& \mbox{ for } x_1\geq 2.\end{array}\right.$$
For $\tau>0$, set $h_\tau(x_1)=h(x_1/\tau)$.
Now, we define a Dirac sequence in $\R^n$:
$$\phi_\epsilon(x):= \psi_\epsilon(x') \frac{\partial h_{\tau(\epsilon)}}{\partial x_1}(x_1),$$
where $\tau(\epsilon)$ will be chosen later. At this point, we only require that $\tau(\epsilon)\leq \epsilon$.
Note that $\supp(\phi_\epsilon)\subset\subset \{x_1>0\}$.
Let
$$f_\epsilon:=f*\phi_\epsilon.$$
Then $f_\epsilon\rightarrow f$ and $Qf_\epsilon\rightarrow Qf$ as in the proof of Theorem \ref{thm:friedrichs}.
Because of \eqref{eq:bvalues2}, we only have to prove that
\begin{eqnarray}\label{eq:bvalues4}
\lim_{\epsilon\rightarrow 0}\ a_1(0,\cdot) f_\epsilon|_{bU} 
= \lim_{\epsilon\rightarrow 0}\ \big(a_1(0,\cdot) f_b\big)*\psi_\epsilon\ \ \mbox{ in }\ L^p(bU).
\end{eqnarray}
For $(0,x')\in bU$, we calculate:
\begin{eqnarray*}
\big(a_1(0,\cdot) f_b\big)*\psi_\epsilon (x')
&=& \int_{bU} a_1(0,t') f_b(t') \psi_\epsilon(x'-t') dt'\\
&=& \int_{bU} a_1(0,t') f_b(t') \psi_\epsilon(x'-t') \big(1-h_{\tau(\epsilon)}(0)\big)dt'\\
&=& \int_U \big(Qf\big) \Phi^\epsilon_{x'} dt - \int_U f \big(Q^*\Phi^\epsilon_{x'}\big)\ dt,
\end{eqnarray*}
if we let
$$\Phi_{x'}^\epsilon(t) := \psi_\epsilon(x'-t') \big( 1 - h_{\tau(\epsilon)}(-t_1)\big)$$
and apply the Green-Stokes formula \eqref{eq:bvalues1}.
By the use of the decomposition \eqref{eq:Q},
it follows that
\begin{eqnarray*}
\big(a_1(0,\cdot) f_b\big)*\psi_\epsilon (x')
&=&  \int_U \big(Qf\big) \Phi_{x'}^\epsilon dt - \int_U f \big(Q' \Phi_{x'}^\epsilon\big)\ dt\\
&+&  \int_U f(t) a_1(t) \phi_\epsilon(x-t) dt
\end{eqnarray*}
with $x=(0,x')$. We will now show that the first two terms turn to $0$ in $L^p(bU)$
if we let $\epsilon\rightarrow 0$.

So, consider:
\begin{eqnarray*}
\int_{bU} \left|\int_U \big(Qf\big) \Phi_{x'}^\epsilon dt\right|^p dx'
&\leq& \int_{bU} \int_{U} |Qf|^p (\Phi_{x'}^\epsilon)^p dt dx'\\
&\leq& \int_{U} |Qf|^p (1-h_{\tau(\epsilon)}(-t_1))\int_{bU} \psi_{\epsilon}(x'-t') dx' dt\\
&=&  \int_{U} |Qf|^p (1-h_{\tau(\epsilon)}(-t_1)) dt
\end{eqnarray*}
Here, $|Qf|^p (1-h_{\tau(\epsilon)}(-t_1)) \leq |Qf|^p$ (which is in $L^1(U)$),
and converges to $0$ point-wise. Hence, the whole expression turns to $0$
by Lebesgue's Theorem. For the second term, note that
\begin{eqnarray*}
\left|Q'_t \Phi^\epsilon_{x'}\right| 
&=& \left|\big(1-h_{\tau(\epsilon)}(-t_1)\big) Q'_t \psi_\epsilon(x'-\cdot)\right|
\lesssim \epsilon^{-n} \left(1-h_{\tau(\epsilon)}(-t_1)\right).
\end{eqnarray*}
Hence, we conclude:
\begin{eqnarray*}
\int_{bU} \left|\int_U f \big(Q' \Phi_{x'}^\epsilon\big) dt\right|^p dx'
&\leq& \int_U |f|^p \left( \int_{bU} |Q' \Phi_{x'}^\epsilon|^p dx'\right) dt\\
&=& \int_U |f|^p \left( \int_{B_\epsilon'(t')} |Q' \Phi_{x'}^\epsilon|^p dx'\right) dt\\
&\lesssim& \frac{1}{\epsilon} \int_U |f|^p \big(1-h_{\tau(\epsilon)}(-t_1)\big) dt.
\end{eqnarray*}
Here now, for fixed $\epsilon>0$, $\epsilon^{-1}|f|^p\in L^1(U)$,
$$\frac{1}{\epsilon} |f|^p \big(1-h_{\tau}(-t_1)\big) \leq \frac{1}{\epsilon} |f|^p,$$
and the left-hand side converges to $0$ point-wise for $\tau\rightarrow0$.
So, by the Theorem of Lebesgue, there exists $\tau(\epsilon)$ such that
$$\frac{1}{\epsilon} \int_U |f|^p \big(1-h_{\tau(\epsilon)}(-t_1)\big) dt \leq \epsilon.$$
This is our choice of $\tau(\epsilon)$ which has been left open before.
So, we have just seen that
$$\lim_{\epsilon\rightarrow 0}
\big(a_1(0,\cdot) f_b\big)*\psi_\epsilon (x')
= \lim_{\epsilon\rightarrow 0} \int_U f(t) a_1(t) \phi_\epsilon(x-t) dt$$
in $L^p(bU)$. Recall that we had reduced the problem to showing \eqref{eq:bvalues4}.
So, only
$$\lim_{\epsilon\rightarrow 0}\int_U f(t) \big( a_1(t)-a_1(x)\big) \phi_\epsilon(x-t) dt =0$$
in $L^p(bU)$ remains to show.
But, due to compactness, there exists a constant $C>0$ such that
$|a_1(t)-a_1(x)|\leq C \epsilon$ if $|t-x|\leq \epsilon$.
Since $f\in L^p(U)$ and $|\phi_\epsilon|\leq 1$,
the proof is finished easily.
\end{proof}

We remark that the assumptions on the regularity of the boundary $bM$ could be relaxed considerably.

\section{Boundary Values for the $\dq$-Operator}\label{sec:dq}

In this section, we will apply Friedrichs' extension lemma with boundary values,
Theorem \ref{thm:main}, to the $\dq$-operator. Recall the following
definition of $\dq$-boundary values that is common in complex analysis:

\begin{defn}\label{defn:dq}
Let $D\subset\subset \C^n$ be a bounded domain with smooth boundary $bD$, and
$f\in L_{0,q}^p(D)$ with $\dq f\in L^p_{0,q+1}(D)$ in the sense of distributions for $1\leq p<\infty$.
Then, we say that $f$ has weak $\dq$-boundary values $f_b\in L^p_q(bD)$ if
\begin{eqnarray}\label{eq:dqbvalues}
\int_D \dq f\wedge \phi + (-1)^q \int f\wedge \dq \phi = \int_{bD} f_b\wedge \iota^*(\phi)
\end{eqnarray}
for all $\phi\in C^\infty_{n,n-q-1}(\o{D})$, where $\iota: bD\hookrightarrow \C^n$
is the embedding of the boundary.
\end{defn}

In fact, the left hand side of \eqref{eq:dqbvalues} depends only on the pull-back $\iota^*(\phi)$
of $\phi$ to $bD$,
and so it defines a current on $bD$. Generally, this current is called the weak $\dq$-boundary value
of $f$, and we say that $f$ has got boundary values in $L^p$, if this current can be represented
by an $L^p$-form as in Definition \ref{defn:dq}. See \cite{He2} for a more detailed treatment of that topic.
Boundary values as in Definition \ref{defn:dq} are not uniquely defined. The reason is as follows:
Let $r\in C^\infty(\C^n)$ be a defining function for $D$. So, $D=\{z\in \C^n: r(z)<0\}$
and we may assume that $\|dr\|\equiv 1$ on $bD$. Then $\iota^*(dr)=0$ implies $\iota^*(\dq r)=-\iota^*(\partial r)$.
Note that $\phi\in C^\infty_{n,n-q-1}(\o{D})$ contains $\partial r$ necessarily.
Hence, $\iota^*(\dq r)\wedge\iota^*(\phi)=0$ on $bD$ for all $\phi\in C^\infty_{n,n-q-1}(\o{D})$.
One should mention an example where weak $\dq$-boundary values occur:
\begin{thm}{\bf (Harvey-Polking \cite{HaPo})}\label{thm:hapo}
Let $r$ be the strictly plurisubharmonic defining function of a strictly pseudoconvex domain $D\subset\subset \C^n$,
and $\omega\in L^1_{0,1}(D)$ with $\dq \omega=0$ and $|r|^{-1/2} \dq r \wedge\omega\in L^1_{0,2}(D)$.
Then there exists $f\in L^1(D)$ with weak $\dq$-boundary values $f_b\in L^1(bD)$ such that $\dq f=\omega$.
\end{thm}
We will now show that Definition \ref{defn:dq} is actually equivalent to Definition \ref{defn:bvalues}
if we make the right choices. So, let $M=\o{D}$ with the underlying Riemannian structure on $\C^n$,
$E=\Lambda^{0,q} T^* M$, $F=\Lambda^{0,q+1} T^* M$, and
$$Q:=\dq: C^\infty_{0,q}(M)=C^\infty(M,E) \longrightarrow C^\infty_{0,q+1}(M)=C^\infty(M,F).$$
Note that $Q^*=-*\partial *$. For $u,v\in C^\infty_{0,q}(\o{D})$:
$$(u,v)_M = \int_M \langle u,v\rangle_E dV_{\C^n} = \int_M u\wedge *\o{v}.$$
In order to reformulate \eqref{eq:dqbvalues}, let $g:=(-1)^{q+1} *\o{\phi} \in C^\infty_{0,q+1}(D)$.
Then $\phi=*\o{g}$. So,
$$\int_D \dq f\wedge\phi = \int_D \dq f\wedge *\o{g} = \int_M \langle\dq f, g\rangle_F\ dV_M = (Qf,g)_M\ ,$$
and
\begin{eqnarray*}
\int_D f\wedge\dq \phi &=& -\int_D f\wedge ** \dq **  \phi = \int_D f\wedge * \o{Q^* *\o{\phi}}\\
&=& (-1)^{q+1} \int_M \langle f,Q^* g\rangle_E\ dV_{\C^n} = (-1)^{q+1} (f,Q^* g)_M.
\end{eqnarray*}
Hence, in the notation of Definition \ref{defn:bvalues}, the left hand side of \eqref{eq:dqbvalues}
reads exactly as
$(Q f, g)_M - (f,Q^*, g)_M$.
For the right hand side, recall that we have chosen the defining function $r$ such that $\|dr\|\equiv 1$
on $bD$. That implies $dS_{bD} = \iota^*(* dr)$. Note that there is a $(0,q)$-form $f_b'\in C^\infty(bD,\Lambda^{0,q} T^*\C^n|_{bD})$
such that $\iota^*(f_b')=f_b$. Sine $\iota^* dr=0$ and $dr\wedge\phi=\dq r\wedge\phi$, we compute
\begin{eqnarray*}
f_b\wedge \iota^*(\phi) &=& \iota^*(f_b'\wedge \phi) 
= \iota^*\big( (*[dr\wedge f_b'\wedge \phi]) *dr\big)
= \big( *[\dq r\wedge f_b'\wedge\phi]\big)\ dS_{bD}\\
&=& \big( *[\dq r\wedge f_b'\wedge * \o{g}]\big) dS_{bD}
= \langle \dq r\wedge f_b', g\rangle_F\ dS_{bM}\\
&=& \langle Q(r f_b'), g\rangle_F\ dS_{bM}
= \frac{1}{i}\langle \sigma_Q(\cdot,\nu^\flat) f_b', g\rangle_F\ dS_{bM}.
\end{eqnarray*}
So, we have
$$\int_{bD} f_b\wedge\iota^*(\phi) = \frac{1}{i}\int_{bM} \langle \sigma_Q(x,\nu^\flat) f_b', g\rangle_{F_x}\ dS_{bM}(x)\ ,$$
and recognize therefore:
\begin{lem}\label{lem:equivalence}
$f\in L^p_{0,q}(D)$ with $\dq f\in L^p_{0,q+1}(D)$ has weak $\dq$-boundary values $f_b\in L^p_q(bD)$
according to Definition \ref{defn:dq} exactly if it has $\dq$-boundary values \linebreak $f_b'\in L^p(bD, \Lambda^{0,q} T^* \C^n|_{bD})$
according to Definition \ref{defn:bvalues}.
\end{lem}

So, we are now in the position to translate Theorem \ref{thm:main}
into the Friedrichs' extension lemma with boundary values for the $\dq$-operator:

\begin{thm}\label{thm:main2}
Let $D\subset\subset \C^n$ be a bounded domain with smooth boundary $bD$, and
$f\in L_{0,q}^p(D)$ with $\dq f\in L^p_{0,q+1}(D)$ in the sense of distributions for $1\leq p<\infty$.
Then $f$ has weak $\dq$-boundary values $f_b\in L^p_q(bD)$ according to Definition \ref{defn:dq}
exactly if there is a sequence $\{f_\epsilon\}$ in $C^\infty_{0,q}(\o{D})$ such that
$\lim_{\epsilon\rightarrow 0} f_\epsilon = f$ in $L^p_{0,q}(D)$,
$\lim_{\epsilon\rightarrow 0} \dq f_\epsilon = \dq f$ in $L^p_{0,q+1}(D)$, and
$$\lim_{\epsilon\rightarrow 0}\ \iota^*(f_\epsilon\wedge \phi) = f_b\wedge \iota^*(\phi)\ \ \mbox{ in }\ L^p_{2n-1}(bD)$$
for all $\phi\in C^\infty_{n,n-q-1}(\o{D})$.
If $r\in C^\infty(\C^n)$ is a defining function for $D$, i.e. \linebreak $D=\{z\in \C^n: r(z)<0\}$
and $dr\neq 0$ on $bD$, then the last condition is equivalent to
$$\lim_{\epsilon\rightarrow 0}\ \iota^*(f_\epsilon\wedge \partial r) = f_b \wedge \iota^*(\partial r)\ \ \mbox{ in }\ L^p_{q+1}(bD).$$
If $q=0$, then this in turn is equivalent to
$$\lim_{\epsilon\rightarrow 0}\ \iota^*(f_\epsilon) = f_b \ \ \mbox{ in }\ L^p(bD).$$
\end{thm}

\newpage
\section{Regularity of the BMK Formula}

The characterization of weak $\dq$-boundary values by approximation is a quite useful tool
because it allows us to simply work in the $C^\infty$-category in many situations.
As an application, we will derive the Bochner-Martinelli-Koppelman formula
for $L^p$-forms with weak $\dq$-boundary values.
Before doing that, we present another technical but useful result.
For convenience of the reader, let us recall shortly the Bochner-Martinelli-Koppelman formula.
\begin{defn}\label{defn:kq}
Let $0\leq q\leq n$. The Bochner-Martinelli-Koppelman kernel $B_{nq}$ in $\C^n$
is then given as
$$B_{nq}(\zeta,z)=\frac{(n-1)!}{2^{q+1}\pi^n} \frac{1}{\|\zeta-z\|^{2n}}
\sum_{\substack{j,J,\\|L|=q+1}} \epsilon^{L}_{jJ} (\o{\zeta_j}-\o{z_j})(*d\zeta^L)\wedge d\o{z}^J,$$
where
$$\epsilon^A_B:=\left\{\begin{array}{ll}\mbox{sign } \pi &
\mbox{, if $A=B$ as sets and } \pi \mbox{ is a permutation with $B=\pi A$, }\\
0 & \mbox{, if $A\neq B$.}
\end{array}\right.$$
Moreover, let $B_{n,-1}\equiv 0$.
\end{defn}
Now, let $D\subset\subset \C^n$ be a bounded domain with $C^1$-smooth boundary $bD$.
If $g$ is a measurable $(0,q+1)$-form on $D$, let
$${\bf B}^D_q g(z):=\int_D g(\zeta)\wedge B_{nq}(\zeta,z),$$
and if $f$ is a measurable $q$-form on $bD$, let
$${\bf B}^{bD}_q f(z):=\int_{bD} f(\zeta)\wedge B_{nq}(\zeta,z),$$
provided, the integrals do exist. Then:
\begin{thm}\label{thm:bmk-formel}
{\bf (BMK formula \cite{Ko})}
Let $D\subset\subset \C^n$ be a bounded domain with $C^1$-smooth boundary $bD$, $1\leq q\leq n$,
and $f\in C^1_{0,q}(\o{D})$. Then:
\begin{eqnarray}\label{eq:bmk}
f(z)={\bf B}^{bD}_q f(z)-{\bf B}^D_q(\dq f)(z)-\dq_z {\bf B}^D_{q-1}f(z),
\end{eqnarray}
where ${\bf B}^D_{q-1}f\in C^1_{0,q-1}(D)$.
\end{thm}
In the following, we will show that \eqref{eq:bmk} is still valid
under the assumption that $f\in L^1_{0,q}(D)$ with $\dq f\in L^1_{0,q+1}(D)$
has weak $\dq$-boundary values $f_b\in L^1_q(bD)$.
It is well-known that

\begin{lem}\label{lem:bmkreg1}
Let $D\subset\subset \C^n$ be a bounded domain. Then,
${\bf B}^D_q$ defines a bounded linear operator
$$L^p_{0,q+1}(D)\rightarrow L^r_{0,q}(D)$$
for all $1\leq p,r\leq\infty$ such that $1/r > 1/p - 1/(2n)$.
\end{lem}

This is a direct consequence of $\|B_{nq}(\zeta,z)\|\lesssim \|\zeta-z\|^{2n-1}$
and Young's inequality,
which is usually used for estimating integral operators (cf. for example \cite{LiMi}, Proposition III.5.35).
In order to estimate the BMK boundary operator ${\bf B}^{bD}_q$,
we need a more general version of such an inequality. So, we will make use of
the following technical result. The proof can be found in \cite{Rp2}, Theorem 3.3.4:

\begin{thm}\label{thm:young2}
Let $1\leq t\leq s<\infty$ and $1\leq a, b\leq\infty$ be fixed,
$(X,\mu)$ and $(Y,\nu)$ measure spaces with
$\mu(X)<\infty$ and $\nu(Y)<\infty$,
and $K$ a $\mu\times\nu$-measurable function on $X\times Y$ such that
\begin{eqnarray}
\int_X |K(x,y)|^t d\mu(x)\leq g(y) & \mbox{for almost all } y \in Y,\label{eq:rand1}\\
\int_Y |K(x,y)|^s d\nu(y)\leq h(x) & \mbox{for almost all } x \in X,\label{eq:rand2}
\end{eqnarray}
where $g\in L^a(Y)$ and $h\in L^b(X)$. Then:

\vspace{1mm}
I. The linear operator $f\mapsto {\bf T}f$ which is given by
$${\bf T}f(y)=\int_X K(x,y)f(x) d\mu(x)$$
for almost all $y\in Y$ defines a bounded operator
${\bf T}: L^p(X) \rightarrow L^r(Y)$ for all
$1\leq p,r\leq\infty$ satisfying
\begin{eqnarray}\label{eq:rand22a}
p\geq \left\{
\begin{array}{lll}
\frac{t}{t-1} & , \mbox{ if } & t>1,\\
\infty & , \mbox{ if } & t=1,
\end{array}
\right.
\end{eqnarray}
and
\begin{eqnarray*}
r\leq at.
\end{eqnarray*}

\vspace{1mm}
II. The mapping $f\mapsto {\bf T} f$ is bounded as an operator
${\bf T}: L^p(X)\rightarrow L^1(Y)$ for $1\leq p < \infty$ with
\begin{eqnarray}\label{eq:rand3}
p\geq \left\{
\begin{array}{lll}
\frac{sb}{sb-1} & , \mbox{ if } & 1<sb<\infty,\\
1 & , \mbox{ if } & b=\infty.
\end{array}
\right.
\end{eqnarray}

\vspace{1mm}
III. If \eqref{eq:rand3} is satisfied and $sb\neq t$, 
then $f\mapsto {\bf T}f$ defines a bounded operator ${\bf T}: L^p(X)\rightarrow L^r(Y)$ 
for all $1\leq r \leq \infty$ with
\begin{eqnarray}\label{eq:rand4}
\frac{1}{r} = \left(\frac{sb}{sb-t}\right)\left(\frac{1}{p}+\frac{1}{t}-1\right)
\end{eqnarray}
and
\begin{eqnarray}\label{eq:rand5}
r \leq t \left(a\frac{s-t}{s}+1\right).
\end{eqnarray}
We have made the following conventions:
In \eqref{eq:rand4}, let $1/r=0$ if $r=\infty$.
If $b=\infty$, then \eqref{eq:rand4} has to be interpreted as $\frac{1}{r} = \frac{1}{p}+\frac{1}{t}-1$.
If $a=\infty$, then \eqref{eq:rand5} reads as $r\leq\infty$.
\end{thm}

It is now easy to deduce:

\begin{lem}\label{lem:bmkreg2}
Let $D\subset\subset \C^n$ be a bounded domain with $C^1$-smooth boundary $bD$.
Then, ${\bf B}^{bD}_q$ is bounded as an operator
$$L^p_q(bD)\rightarrow L^p_{0,q}(D)$$
for all $1\leq p < \infty$.
\end{lem}

\begin{proof}
We will apply Theorem \ref{thm:young2} to the operator ${\bf B}^{bD}_q$.
So, let $X=bD$, $Y=D$ and
$$|K(x,y)|=|B_{nq}(x,y)| \leq \frac{A}{|x-y|^{2n-1}},$$
where $A>0$ is a constant that depends only on $D$, $q$ and $n$.
We choose $t=1$. It is not hard to prove that there are constants $C_0(D)>0$
and $C_1(D)>0$ such that
$$\int_X |K(x,y)|^t d\mu(x) \leq C_0(D) + C_1(D) |\log \delta(y)| =: g(y),$$
where 
$$\delta(y):=\dist (y, bD).$$
For a proof, we refer to \cite{Rp2}, Lemma 3.3.1.
It is easy to see that $|g|^a$ is integrable over $Y=D$ for all powers $1\leq a< \infty$.
So, we remark that $g\in L^a(Y)$ for all $1\leq a<\infty$ (cf. \cite{Rp2}, Lemma 3.3.3).
Now, choose $s>1$ such that
$$1=t < s < \frac{2n}{2n-1}.$$
Then
$$ h(x) := \int_{Y} |K(x,y)|^s d\nu(y)$$
is uniformly bounded (independent of $x\in X$).
Hence $h\in L^\infty(X)$.
So, the assumptions of Theorem \ref{thm:young2} are fulfilled for
$X=bD$, $Y=D$, ${\bf T}={\bf B}^{bD}_q$, $1=t<s$, $h\in L^\infty(X)$, i.e. $b=\infty$,
and $g\in L^a(Y)$ for all $1\leq a<\infty$.
We conclude that ${\bf B}^{bD}_q$ defines a bounded linear operator
${\bf B}^{bD}_q: L^p_q(bD)\rightarrow L^r_{0,q}(D)$
for all $1\leq p,r <\infty$ such that
$$\frac{1}{r}=\frac{1}{p}+\frac{1}{t}-1=\frac{1}{p}.$$
\end{proof}

We have now provided all the tools that are needed to derive the
Bochner-Martinelli-Koppelman formula for $L^p$-forms with weak $\dq$-boundary
values as an application of Friedrichs' extension theorem with boundary values.
So, let \linebreak $D\subset\subset \C^n$ be a bounded domain with smooth boundary, $1\leq r, p <\infty$,\linebreak
$f\in L^p_{0,q}(D)$ with $\dq f\in L^r_{0,q+1}(D)$ in the sense of distributions and assume
that $f$ has weak $\dq$-boundary values $f_b\in L^p_q(bD)$ according to Definition \ref{defn:dq}.

Then, by Theorem \ref{thm:main2}, it follows that
there exists a sequence $\{f_\epsilon\}$ in $C^\infty_{0,q}(\overline{D})$ such that
\begin{eqnarray*}
\lim_{\epsilon\rightarrow 0} f_\epsilon &=& f\ \mbox{ in }\ L^1_{0,q}(D),\\
\lim_{\epsilon\rightarrow 0} \dq f_\epsilon &=& \dq f\ \mbox{ in }\ L^1_{0,q+1}(D),
\end{eqnarray*}
and
\begin{eqnarray}\label{eq:phiX}
\lim_{\epsilon\rightarrow 0} f_\epsilon|_{bD}\wedge\iota^*(\phi)
=\lim_{\epsilon\rightarrow 0} \iota^*(f_\epsilon\wedge\phi)
=f_b\wedge\iota^*(\phi)\ \mbox{ in }\ L^1_{2n-1}(bD)
\end{eqnarray}
for all $\phi\in C^\infty_{n,n-q-1}(\o{D})$,
where $\iota: bD \rightarrow \C^n$ denotes the embedding.
In the following, we will simply write $f_\epsilon$ instead of $f_\epsilon|_{bD}$.\hfill\vspace{4mm}\linebreak
Now, the classical BMK formula, Theorem \ref{thm:bmk-formel}, implies:
$$f_\epsilon(z)={\bf B}^{bD}_q f_\epsilon(z)-{\bf B}^D_q(\dq f_\epsilon)(z)-\dq_z {\bf B}^D_{q-1}f_\epsilon(z)$$
for all $z\in D$, which we permute to:
\begin{eqnarray}\label{eq:bmke}
\dq_z {\bf B}^D_{q-1}f_\epsilon(z)={\bf B}^{bD}_q f_\epsilon(z)-{\bf B}^D_q(\dq f_\epsilon)(z)-f_\epsilon(z).
\end{eqnarray}\linebreak
By Lemma \ref{lem:bmkreg1} and Lemma \ref{lem:bmkreg2}, we know that the applications
\begin{eqnarray*}
{\bf B}^{bD}_q: L^1_q(bD) &\rightarrow& L^1_{0,q}(D),\\
{\bf B}^D_q: L^1_{0,q+1}(D) &\rightarrow& L^1_{0,q}(D)
\end{eqnarray*}
are continuous. Hence, the right hand side of \eqref{eq:bmke} converges in $L^1_{0,q}(D)$ to a form
\begin{eqnarray}\label{eq:g}
G={\bf B}^{bD}_q f_b-{\bf B}^D_q(\dq f)-f\ \ \ \in L^1_{0,q}(D).
\end{eqnarray}
To see this, note that the Bochner-Martinelli-Koppelman kernel $B_{nq}(\zeta,z)$
is a $(n,n-q-1)$-form in $\zeta$. So, \eqref{eq:phiX} can be used.\hfill\vspace{2mm}\linebreak
Since
\begin{eqnarray*}
\lim_{\epsilon\rightarrow 0} {\bf B}^D_{q-1}f_\epsilon &=& {\bf B}^D_{q-1} f\ \mbox{ in }\ L^1_{0,q-1}(D),\\
\lim_{\epsilon\rightarrow 0} \dq_z {\bf B}^D_{q-1}f_\epsilon &=& G\ \mbox{ in }\ L^1_{0,q}(D),
\end{eqnarray*}
it follows that $G$
actually is the $\dq$-derivate in the sense of distributions:
$$G=\dq_z {\bf B}^D_{q-1} f.$$
Applying Lemma \ref{lem:bmkreg1} and Lemma \ref{lem:bmkreg2} again, we observe that
\begin{eqnarray*}
{\bf B}^{bD}_q f_b &\in& L^p_{0,q}(D),\\
{\bf B}^D_q(\dq f) &\in& L^r_{0,q}(D).
\end{eqnarray*}
So, the right hand side of \eqref{eq:g}, and therefore $G$, is in $L^r_{0,q}(D) \cap L^p_{0,q}(D)$.

\newpage
We summarize:
\begin{thm}{\bf (BMK formula for $L^p$-forms)}\label{thm:bmklp}
Let $D\subset\subset \C^n$ be a bounded domain with smooth boundary and $0\leq q\leq n$. 
Moreover, let $1\leq r,p < \infty$ and
$f\in L^p_{0,q}(D)$
with
$\dq f\in L^r_{0,q+1}(D)$,
such that $f$ has weak $\dq$-boundary values $f_b\in L^p_q(bD)$.
Then
\begin{eqnarray*}
{\bf B}^{bD}_q f_b &\in& L^p_{0,q}(D),\\
{\bf B}^D_q(\dq f) &\in& L^r_{0,q}(D),\\
{\bf B}^D_{q-1}f &\in& L^p_{0,q-1}(D)\cap Dom(\dq),\\
\dq {\bf B}^D_{q-1}f &\in& L^r_{0,q}(D) \cap L^p_{0,q}(D),
\end{eqnarray*}
and
\begin{eqnarray*}\label{eq:bmklp}
f(z)={\bf B}^{bD}_q f_b(z)-{\bf B}^D_q(\dq f)(z)-\dq_z {\bf B}^D_{q-1}f(z)
\end{eqnarray*}
for almost all $z\in D$.
\end{thm}


\end{document}